\documentclass[12pt]{article}
\usepackage{amsmath,amssymb,amsthm}
\usepackage[mathscr]{euscript}
\usepackage[letterpaper]{geometry}
\usepackage{comment}
\usepackage{graphicx}
\usepackage{latexsym}
\usepackage{color}

\def\BBox{\kern  -0.16cm\hbox{\vrule width 0.16cm height 0.16cm}}

\renewcommand{\bar}{\overline}

\renewcommand{\leq}{\leqslant} 
\renewcommand{\geq}{\geqslant}

\newcommand{\CP}{\mathcal{P}}
\newcommand{\K}{\mathcal{K}} 
\newcommand{\CL}{\mathcal{L}}

\newtheorem{lemma}{Lemma}[section]
\newtheorem{theorem}{Theorem}[section]
\newtheorem{problem}{Problem}[section]

\newtheorem{corollary}{Corollary}[section]

\title{Regular Incidence Complexes, Polytopes, and C-Groups}
\author{
Egon Schulte\thanks{Email: schulte@neu.edu}\\[.03in]
Department of Mathematics\\[.03in]
Northeastern University, Boston, MA 02115, USA}

\begin{document}
\maketitle

\begin{abstract}
\noindent
Regular incidence complexes are combinatorial incidence structures generalizing regular convex polytopes, regular complex polytopes, various types of incidence geometries, and many other highly symmetric objects. The special case of abstract regular polytopes has been well-studied. The paper describes the combinatorial structure of a regular incidence complex in terms of a system of distinguished generating subgroups of its automorphism group or a flag-transitive subgroup. Then the groups admitting a flag-transitive action on an incidence complex are characterized as generalized string C-groups. Further, extensions of regular incidence complexes are studied, and certain incidence complexes particularly close to abstract polytopes, called abstract polytope complexes, are investigated. 
\end{abstract}




\section{Introduction}

Regular incidence complexes are combinatorial incidence structures with very high combinatorial symmetry. The concept was introduced by Danzer~\cite{dan,kom1} building on Gr\"unbaum's~\cite{grgcd} notion of a polystroma. Regular incidence complexes generalize regular convex polytopes~\cite{crp}, regular complex polytopes~\cite{cox1,shcp}, various types of incidence geometries~\cite{buc,bup,leem,tib}, and many other highly symmetric objects. The terminology and notation is patterned after convex polytopes~\cite{gcp} and was ultimately inspired by Coxeter's work on regular figures~\cite{crp,cox1}. The first systematic study of incidence complexes from the discrete geometry perspective occurred in~\cite{esdiss} and the related publications~\cite{kom1,kom2,kom3,onarr}.

The special case of abstract polytopes has recently attracted a lot of attention (see McMullen \& Schulte~\cite{arp}). Abstract polytopes (or incidence polytopes, as they were called originally) are incidence complexes close to ordinary polytopes and are in a sense topologically real.

Incidence complexes can also be viewed as incidence geometries or diagram geometries with a linear diagram (see Buekenhout-Cohen~\cite{buc}, Buekenhout-Pasini~\cite{bup}, Leemans~\cite{leem} and Tits~\cite{tib}), although here we study these structures from the somewhat different discrete geometric and combinatorial perspective of polytopes and ranked partially ordered sets.

The present paper is organized as follows. In Section~\ref{inccom}, we introduce incidence complexes following the original definition of \cite{kom1} with minor amendments. Then in Sections~\ref{groupfcom} and~\ref{comfgroups} we derive structure results for flag-transitive subgroups of regular incidence complexes and characterize these groups as what we will call here generalized C-groups, basically following \cite{esdiss,kom2} (apart from minor changes inspired by \cite{arp}). Section~\ref{regpols} explains how abstract regular polytopes fit into the more general framework of regular incidence complexes. In Section~\ref{ext} we discuss extensions of regular incidence complexes. Section~\ref{apcs} is devoted to the study of abstract polytope complexes, a particularly interesting class of regular incidence complexes which are not abstract polytopes but still relatively close to abstract polytopes. This section also describes a number of open research problems. Finally, Section~\ref{notes} collects historical notes on incidence complexes and some personal notes related to the author's work.

\section{Incidence Complexes}
\label{inccom}

Following \cite{kom1,esdiss}, an {\em incidence complex $\K$ of rank $n$\/}, or simply an {\em $n$-complex\/}, is a partially ordered set (poset), with elements called {\em faces\/}, which has the properties (I1),\ldots,(I4) described below. \smallskip

\noindent
\textbf{(I1)}\  $\K$ has a least face $F_{-1}$ and a greatest face $F_n$, called the {\it improper\/} faces. All other faces of $\K$ are {\it proper\/} faces of $\K$.
\smallskip

\noindent
\textbf{(I2)}\  Every totally ordered subset, or {\em chain\/}, of $\K$ is contained in a (maximal) totally ordered subset of $\K$ with exactly $n+2$ elements, called a {\em flag\/} of $\K$.
\smallskip

The conditions (I1) and (I2) make $\K$ into a {\it ranked} partially ordered set with a strictly monotone rank function with range $\{-1,0,\ldots,n\}$. A face of rank $i$ is called an $i$-\textit{face}.  A face of rank $0$, $1$ or $n-1$ is also called a {\em vertex\/}, an {\em edge\/} or a {\em facet}, respectively. The faces of $\K$ of ranks $-1$ and $n$ are $F_{-1}$ and $F_n$, respectively. The {\it type\/} of a chain of $\K$ is the set of ranks of faces in the chain. Thus each flag has type $\{-1,0,\ldots,n\}$; that is, each flag $\Phi$ of $\K$ contains a face of $\K$ of each rank $i$ with $i=-1,0,\ldots,n$. 

For an $i$-face $F$ and a $j$-face $G$ of $\K$ with $F \leq G$ we call
\[ G/F := \{ H \in \K \, | \, F \leq H \leq G \} \]
a \textit{section} of $\K$. This will be an incidence complex in its own right, of rank $j-i-1$. Usually we identify a $j$-face $G$ of $\K$ with the $j$-complex $G/F_{-1}$. Likewise, if $F$ is an $i$-face, the $(n-i-1)$-complex $F_n/F$ is called the {\em co-face\/} of $F$ in $\K$, or the \textit{vertex-figure} at $F$ if $F$ is a vertex.

A partially ordered set $\K$ with properties (I1) and (I2) is said to be {\em connected\/} if either $n \leq 1$, or $n \geq 2$ and for any two proper faces $F$ and $G$ of $\K$ there exists a finite sequence of proper faces $F = H_{0},H_{1},\ldots,H_{k-1},H_{k} = G$ of $\K$ such that $H_{j-1}$ and $H_{j}$ are incident for $j = 1,\ldots,k$.  We say that $\K$ is {\em strongly connected\/} if each section of $\K$ (including $\K$ itself) is connected. 

\noindent
\textbf{(I3)}\  $\K$ is {\em strongly connected.\/}

\noindent
\textbf{(I4)}\  For each $i = 0,1,\ldots,n-1$, if $F$ and $G$ are incident faces of
$\K$, of ranks $i-1$ and $i+1$ respectively, then there are {\em at least 
two\/} $i$-faces $H$ of $\K$ such that $F < H < G$.

Thus, an $n$-complex $\K$ is a partially ordered set with properties (I1),\ldots,(I4).

An {\em abstract $n$-polytope\/}, or briefly {\em $n$-polytope\/}, is an incidence complex of rank $n$ satisfying the following condition (I4P), which is stronger than (I4):

\noindent
\textbf{(I4P)}\  For each $i = 0,1,\ldots,n-1$, if $F$ and $G$ are incident faces of
$\K$, of ranks $i-1$ and $i+1$ respectively, then there are {\em exactly 
two\/} $i$-faces $H$ of $\K$ such that $F < H < G$.

We call two flags of $\K$ {\em adjacent\/} if one differs from the other in exactly one face; if this face has rank $i$, with $i=0,\ldots,n-1$, the two flags are {\em $i$-adjacent\/}. Then the conditions (I4) and (I4P) are saying that each flag has at least one or exactly one $i$-adjacent flag for each~$i$, respectively. We refer to (I4P) as the {\it diamond condition\/} (for polytopes). 

Though the above definitions of connectedness and strong connectedness are satisfactory from an intuitive point of view, in practice the following equivalent definitions in terms of flags are more useful. 

A partially ordered set $\K$ with properties (I1) and (I2) is called {\em flag-connected\/} if any two flags $\Phi$ and $\Psi$ of $\K$ can be joined by a sequence of flags $\Phi=\Phi_0,\Phi_1,\ldots,\Phi_{k-1},\Phi_{k}=\Psi$ such that successive flags are adjacent. Further, $\K$ is said to be {\em strongly flag-connected\/} if each section of $\K$ (including $\K$ itself) is flag-connected. It can be shown that $\K$ is strongly flag-connected if and only if any two flags $\Phi$ and $\Psi$ of $\K$ can be joined by a sequence of flags $\Phi=\Phi_0,\Phi_1,\ldots,\Phi_{k-1},\Phi_{k}=\Psi$, all containing $\Phi\cap\Psi$, such that successive flags are adjacent. 

It turns out that a partially ordered set $\K$ with properties (I1) and (I2) is strongly flag-connected if and only if $\K$ is strongly connected. Thus in place of (I3) we could have required the following equivalent condition 

\noindent
\textbf{(I3')}\  $\K$ is {\em strongly flag-connected.\/}

A bijection $\varphi: \K\rightarrow \CL$ from a complex $\K$ to a complex $\CL$ is called an {\em isomorphism} if $\varphi$ is order-preserving (in both directions); that is, $F\leq G$ in $\K$ if and only if $F\varphi \leq G\varphi$ in~$\CL$. An {\em automorphism\/} of a complex $\K$ is an isomorphism from $\K$ to itself. The group of all automorphisms $\Gamma(\K)$ of a complex $\K$ is called the \textit{automorphism group} of $\K$. 

A complex $\K$ is said to be \textit{regular} if $\Gamma(\K)$ is transitive on the flags of $\K$. The automorphism group of a regular complex may or may not be simply transitive on the flags. However, if $\K$ is a polytope then $\Gamma(\K)$ is simply transitive on the flags. 

\begin{lemma}
\label{isomsect}
Let $\K$ be a regular $n$-complex. Then all sections of $\K$ are regular complexes, and any two sections
which are defined by faces of the same ranks are isomorphic.  In particular, $\K$ has isomorphic facets and isomorphic vertex-figures.  
\end{lemma}

\begin{proof}
Let $F$ be an $i$-face and $G$ a $j$-face of $\K$ with $F<G$. Let $\Omega$ denote a chain of $\K$ of type 
$\{-1,0,\ldots,i-1,i,j,j+1,\ldots,n\}$ containing $F$ and $G$. Now consider the action of the stabilizer of $\Omega$ in $\Gamma(\K)$ induced on the section $G/F$ of $\K$. Since $\K$ is a regular complex and the flags of $G/F$ are just the restrictions of flags of $\K$ to $G/F$, this stabilizer is a group that acts flag-transitively on $G/F$ (but not necessarily faithfully). Thus the $(j-i-1)$-complex $G/F$ is regular and its automorphism group $\Gamma(G/F)$ contains, as a flag-transitive subgroup, a quotient of the stabilizer of $\Omega$ in $\Gamma(\K)$. (Unlike for regular polytopes this quotient may be proper.)  

Let $F'$ and $G'$ be another pair of $i$-face and $j$-face with $F'<G'$, and let $\Psi'$ be a flag of $\K$ containing $F'$ and $G'$. Then each automorphism of $\K$ mapping $\Psi$ to $\Psi'$ induces an isomorphism from $G/F$ to $G'/F'$. Thus $G'/F'$ is isomorphic to $G/F$.
\end{proof}

\section{Flag-transitive subgroups of the automorphism group}
\label{groupfcom} 

In this section we establish structure results for flag-transitive subgroups $\Gamma$ of the automorphism group $\Gamma(\K)$ of a regular complex $\K$. We follow \cite[Sect. 2]{kom2} (and \cite{esdiss}) and show that any such group (including $\Gamma(\K)$ itself) has a distinguished system of generating subgroups obtained as follows. For corresponding results for regular polytopes see \cite[Ch. 2B]{arp}.

Throughout this section let $\K$ be a regular $n$-complex, with $n\geq 1$, and let $\Gamma$ be a flag-transitive subgroup of $\Gamma(\K)$. Define $N:=\{-1,0,\ldots,n\}$ and for $J\subseteq N$ set $\overline{J}:=N\setminus J$. Let $\Phi:=\{F_{-1},F_0,\ldots,F_n\}$ be a fixed, or {\em base flag\/}, of $\K$, where~$F_i$ designates the $i$-face in $\Phi$ for each $i\in N$. For each $\Omega\subseteq \Phi$ let $\Gamma_\Omega$ denote the stabilizer of $\Omega$ in $\Gamma$. In particular, $\Gamma_\Phi$ is the stabilizer of the base flag $\Phi$, and $\Gamma_\emptyset = \Gamma$. For $i\in N$ define the subgroup $R_i$ of $\Gamma$ as
\begin{equation}
\label{defri}
R_{i} :=  \Gamma_{\Phi\setminus\{F_i\}} = \langle \varphi\in \Gamma \mid F_j\varphi =F_j \mbox{ for all } j\neq i\rangle .
\end{equation}
Then each $R_i$ contains $\Gamma_\Phi$ as a subgroup, and $R_{-1}=\Gamma_\Phi=R_{n}$. 

For $i=0,\ldots,n-1$ let $k_{i}$ denote the number of $i$-faces of $\K$ in a section $G/F$, where $F$ is an $(i-1)$-face and $G$ an $(i+1)$-face with $F<G$; since $\K$ is regular, this number is independent of the choice of $F$ and $G$. Then
\begin{equation}
\label{kis}
k_{i} = |R_{i}:R_{-1}| \quad (i=0,\ldots,n-1). 
\end{equation} 
Note that each flag of $\K$ has exactly $k_{i}-1$ flags $i$-adjacent to it for each $i$.  

If $\K$ is a regular polytope then $R_{i}$ is generated by an involution $\rho_i$ for $i=0,\ldots,n-1$, and the subgroups $R_{-1}$ and $R_{n}$ are trivial. In this case $k_{i}=2$ for each $i$.

Our first goal is to describe the stabilizers of the subchains of the base flag $\Phi$. For $J\subseteq N$ set 
$\Phi_{J} := \{F_{j} \in \Phi \mid j \in J\}$. 

\begin{lemma}
\label{regstab}
For $J\subseteq N$ we have $\Gamma_{\Phi_J}=\langle R_{j}\mid j\in \overline{J}\,\rangle$. 
\end{lemma}

\begin{proof}
Let $\Lambda:=\langle R_{j}\mid j\in \overline{J}\,\rangle$. It is clear that $\Lambda$ is a subgroup of $\Gamma_{\Phi_J}$, since each subgroup $R_j$ with $j\in \overline{J}$ stabilizes $\Phi_J$. To prove equality of the two groups, note first that $\Gamma_{\Phi_J}$ acts transitively on the set of all flags $\Psi$ of $\K$ with $\Phi_{J} \subseteq \Psi$.  Hence, since the base flag stabilizer $\Gamma_\Phi$ lies in $\Gamma_{\Phi_J}$, it suffices to show that $\Lambda$ also acts transitively on these flags. 

Let $\Psi$ be a flag with $\Phi_{J} \subseteq \Psi$. We show that $\Psi$ lies in the orbit of $\Phi$ under $\Lambda$. Choose a sequence of flags 
\[  \Phi = \Phi_{0},\Phi_{1},\dots,\Phi_{k-1},\Phi_{k} = \Psi,  \]
all containing $\Phi_{J}$, such that successive flags are adjacent.  We proceed by induction on~$k$, the case $k = 0$ being trivial.  By the inductive hypothesis, there exists $\psi \in \Lambda$ such that $\Phi\psi = \Phi_{k-1}$. We know that $\Phi_{k-1}$ and $\Psi = \Phi_{k}$ are $j$-adjacent flags for some $j$, so $\Phi=\Phi_{k-1}\psi^{-1}$ and $\Phi_k\psi^{-1}$ are also $j$-adjacent. By the flag-transitivity of $\Gamma$ there exists an element $\tau\in R_j$ such that $\Phi_k\psi^{-1}=\Phi\tau$ and hence $\Psi = \Phi\tau\psi$. But $j \notin J$, since $\Phi_{J} \subseteq\Phi_{i}$ for each $i$, so $\tau,\psi\in\Lambda$ and hence $\tau\psi\in\Lambda$. Thus $\Psi$ lies in the orbit of $\Phi$ under $\Lambda$.
\end{proof}

As the subgroups $R_{-1}$ and $R_n$ lie in $R_j$ for each $j$, the previous lemma with $J=\emptyset$ immediately implies 

\begin{lemma}
\label{disting}
$\Gamma = \langle R_{-1},R_0,\ldots,R_{n} \rangle = \langle R_0,\ldots,R_{n-1} \rangle$.
\end{lemma}

The subgroups $R_{-1},R_{0},\ldots,R_{n}$ of $\Gamma$ are called the {\it distinguished generating subgroups\/} of~$\Gamma$ ({\it with respect to\/} $\Phi$). 

For each $I\subseteq N$, $I\neq \emptyset$, define the subgroup $\Gamma_{I} := \langle R_i \mid i\in I\rangle$. For $I=\emptyset$ we set $\Gamma_{\emptyset} := R_{-1}$. Then by Lemma~\ref{regstab}, 
\begin{equation}
\label{stab1}
\Gamma_{I} = \Gamma_{\{F_i\mid\, i\in \overline{I}\,\}} = \Gamma_{\Phi_{\,\overline{I}}} \quad (I\subseteq N);
\end{equation}
or equivalently, 
\begin{equation}
\label{stab2}
\Gamma_\Omega = \Gamma_{\{i\mid F_i \not\in\Omega\}} \qquad (\Omega\subseteq \Phi). 
\end{equation}
The subgroups $\Gamma_I$, with $I\subseteq N$, are called the {\em distinguished subgroups\/} of $\Gamma$ ({\it with respect to\/}~$\Phi$). Note that the notation $\Gamma_\emptyset$ can have two meanings, namely as $\Gamma_I$ with $I=\emptyset$, and as $\Gamma_\Omega$ with $\Omega=\emptyset$. The intended meaning should be clear from the context. 

The distinguished subgroups satisfy the following important {\it intersection property\/} (with respect to the distinguished generating subgroups):

\begin{lemma}
\label{intprop}
For $I,J\subseteq N$ we have $\Gamma_I \cap \Gamma_J = \Gamma_{I\cap J}$.
\end{lemma}

\begin{proof}
This follows from the fact that the subgroups involved are stabilizers of subchains of $\Phi$, as expressed in 
equation~(\ref{stab1}). In fact, 
\[ \Gamma_I \cap \Gamma_J 
= \Gamma_{\Phi_{\,\overline{I}}} \cap \Gamma_{\Phi_{\,\overline{J}}}
= \Gamma_{\Phi_{\,\overline{I}}\, \cup\, \Phi_{\,\overline{J}}}
= \Gamma_{\Phi_{\,\overline{I\cap J}}}
= \Gamma_{I\cap J} .\]
Thus the lemma follows.
\end{proof}

When $n\geq 2$ we often omit $R_{-1}$ and $R_n$ from the system of generating subgroups and also refer to 
$R_{0},\ldots,R_{n-1}$ as the {\it distinguished generating subgroups\/}. In fact, in this case Lemma~\ref{intprop} shows that 
\[R_{-1} = R_{n}= R_{0}\cap \ldots \cap R_{n-1},\]
so $R_{-1}$ and $R_n$ are completely determined by $R_0,\ldots,R_{n-1}$. However, for the system of generating subgroups to also permit a characterization of the combinatorial structure of $\K$ when $n=1$, the two subgroups $R_{-1}$ and $R_{n}=R_1$ (with $R_{-1}=R_1$) must be included in the system; that is, $R_{-1}$ and $R_1$ are not determined by $R_{0}$ alone. 

The distinguished generating subgroups have the following commuting properties, which hold at the level of groups, but not generally at the level of elements. 

\begin{lemma}
\label{commu}
For $-1\leq i < j-1 \leq n-1$ we have $R_i R_j = \langle R_i,R_j \rangle = R_j R_i$. 
\end{lemma}

\begin{proof}
This is trivial when $i=-1$ or $j=n$. Now suppose $-1< i < j-1 < n-1$. Clearly, it suffices to show that 
$\langle R_i,R_j \rangle=R_i R_j $. Here the inclusion $\supseteq$ is trivial. 

To establish the opposite inclusion let $\varphi\in \langle R_i,R_j \rangle$. Then $\varphi$ fixes $F_k$ for each $k\neq i,j$, since both $R_{i}$ and $R_j$ fix $F_k$.  Hence, since $i<j-1$, there exists an element $\psi\in R_i$ such that $F_{i}\psi=F_{i}\varphi$. Then $F_{k}\varphi\psi^{-1}=F_k$ for each $k\neq j$. But then there also exists an element $\tau\in R_j$ such that $F_{j}\tau=F_{j}\varphi\psi^{-1}$. Hence $F_{k}\varphi\psi^{-1}\tau^{-1} = F_k$ for each $k$, and therefore $\varphi\psi^{-1}\tau^{-1}\in \Gamma_\Phi$.  But  
$\Gamma_{\Phi}$ is a subgroup of $R_j$, so we have
\[\varphi\in\Gamma_{\Phi}\tau\psi \subseteq \Gamma_{\Phi} R_j R_i = R_jR_i,\] 
as required. This completes the proof
\end{proof}

The commuting properties of Lemma~\ref{commu} do not generally extend to the case when $j=i+1$. It can be shown that if $\K$ is a lattice, then 
\begin{equation}
\label{latri}
R_{i}R_{i+1} \cap R_{i+1}R_i \,=\, R_{i} \cup R_{i+1} \quad (i=0,\dots,n-2). 
\end{equation}
Recall that a lattice is a partially ordered set in which every two elements have a supremum (a least upper bound) and an infimum (a greatest lower bound)~\cite{st}.

We introduce some further notation. For each $i\in N$ we write 
\[ \Gamma_{i}   := \Gamma_{N\setminus \{i\}} = \langle R_j \mid j\neq i\rangle \]
and
\[\begin{array}{rllll}
\Gamma_{i}^{-} &:= &\Gamma_{\{-1,0,\ldots,i\}} &\!\!=& \langle R_j \mid j\leq i\rangle, \\[.03in]
\Gamma_{i}^{+}&:= &\Gamma_{\{i,\ldots,n\}}    &\!\!=& \langle R_j \mid j\geq i\rangle. 
\end{array} \]
Note that $\Gamma_{-1}=\Gamma_{n}=\Gamma$. As an immediate consequence of the commutation rules of Lemma~\ref{commu} we have 
\begin{equation}
\label{gamij}
\Gamma_{i}^{-}\Gamma_{j}^{+} = \Gamma_{j}^{+}\Gamma_{i}^{-}\qquad (-1\leq i < j-1\leq n-1). 
\end{equation}
Further, for each $i\in N$, 
\begin{equation}
\label{gami}
\Gamma_{i} = \Gamma_{i-1}^{-}\Gamma_{i+1}^{+} = \Gamma_{i+1}^{+}\Gamma_{i-1}^{-}.
\end{equation}

Observe that when $-1\leq i\leq j\leq n$ the distinguished subgroup $\langle R_{i+1},\ldots,R_{j-1}\rangle$ of $\Gamma$ 
acts flag-transitively (but generally not faithfully) on the section $F_{j}/F_i$ of $\K$ between the base $i$-face and the base $j$-face. The quotient of $\langle R_{i+1},\ldots,R_{j-1}\rangle$ defined by the kernel of this action is a (generally proper) flag-transitive subgroup of $\Gamma(F_{j}/F_{i})$. In particular, $\Gamma_{i-1}^{-}$ acts flag-transitively on the base $i$-face $F_{i}/F_{-1}$ of $\K$,  and $\Gamma_{i+1}^{+}$ acts flag-transitively on the co-face $F_{n}/F_{i}$ of the base $i$-face of $\K$. 

Our next goal is the characterization of the structure of a regular complex $\K$ in terms of the distinguished generating subgroups $R_{-1},R_{0},\ldots,R_{n}$ of the chosen flag-transitive subgroup $\Gamma$ of $\Gamma(\K)$. By the transitivity properties of $\Gamma$ we can write each $i$-face of $\K$ in the form $F_{i}\varphi$ with 
$\varphi \in \Gamma$. We begin with a lemma.

\begin{lemma}
\label{inclem}
Let $0 \leq i \leq j \leq n-1$, and let $G_{i}$ be an $i$-face of $\K$.  Then $G_{i} \leq F_{j}$ if and only if $G_{i} = F_{i}\gamma$ for some $\gamma \in\Gamma_{j}$.
\end{lemma}

\begin{proof}
If $G_{i} = F_{i}\gamma$ with $\gamma \in \Gamma_{j}$, then $G_{i} \leq F_{j}\gamma =
F_{j}$, as claimed.  For the converse, let $\Psi$ be any flag of $\K$ such
that $\{G_{i},F_{j}\} \subseteq \Psi$.  Then, by Lemma~\ref{regstab},
$F_{j} \in \Phi \cap \Psi$ implies that $\Psi = \Phi\gamma$ for some $\gamma \in
\Gamma_{\{F_{j}\}} = \Gamma_{j}$.  Thus $G_{j} = F_{j}\gamma$, as required.
\end{proof}

We now have the following characterization of the partial order in $\K$.

\begin{lemma}
\label{incchar}
Let $0 \leq i \leq j \leq n-1$, and let $\varphi,\,\psi \in \Gamma$.  Then the
following three conditions are equivalent:\\[.05in]
(a)\ $F_{i}\varphi \leq F_{j}\psi$;\\[.03in]
(b)\ $\varphi\psi^{-1} \in \Gamma_{i+1}^{+}\Gamma_{j-1}^{-}$;\\[.03in]
(c)\ $\Gamma_{i}\varphi \cap \Gamma_{j}\psi \neq \emptyset$.
\end{lemma}

\begin{proof}
We shall prove the equivalence in the form (a) $\Rightarrow$ (c) $\Rightarrow$ (b) $\Rightarrow$ (a).

Assume that (a) holds. Then $F_{i}\varphi\psi^{-1} \leq F_{j}$, and thus $F_{i}\varphi\psi^{-1}=F_{i}\gamma$ for some $\gamma\in\Gamma_j$ by Lemma~\ref{inclem}.  In turn, this says that $\Gamma_{i}\gamma\cap \Gamma_{j} \neq \emptyset$, since $\gamma$ lies in this intersection. But $(\varphi\psi^{-1})\gamma^{-1}\in\Gamma_{\{F_i\}}=\Gamma_i$, so $\Gamma_{i}\varphi\psi^{-1}=\Gamma_{i}\gamma$. Hence $\Gamma_{i}\varphi \cap \Gamma_{j}\psi \neq \emptyset$. Thus (c) holds.

If (c) holds, the commuting properties of Lemma~\ref{commu} show that
\begin{equation}
\label{gami+1j-1}
\begin{array}{lll}
\varphi\psi^{-1} \in \Gamma_{i}\Gamma_{j} & = & 
\Gamma_{i+1}^{+}\Gamma_{i-1}^{-}\Gamma_{j-1}^{-}\Gamma_{j+1}^{+} \\
& = & \Gamma_{i+1}^{+}\Gamma_{j-1}^{-}\Gamma_{j+1}^{+} \\
& = & \Gamma_{i+1}^{+}\Gamma_{j+1}^{+}\Gamma_{j-1}^{-} \\
& = & \Gamma_{i+1}^{+}\Gamma_{j-1}^{-}, 
\end{array}
\end{equation}
as required for (b).

Finally, suppose (b) holds.  Then $\varphi\psi^{-1} = \alpha\beta$ for some $\alpha \in \Gamma_{i+1}^{+}$ and
$\beta \in \Gamma_{j-1}^{-}$.  We deduce that
\[  F_{i}\varphi\psi^{-1} = F_{i}\alpha\beta = F_{i}\beta \leq F_{j}\beta = F_{j},\]
so that $F_{i}\varphi \leq F_{j}\psi$, which is (a).  This completes the proof.
\end{proof}

The previous lemma has important consequences. In effect, it says that we may identify a face $F_{i}\varphi$ of   a regular complex $\K$ with the right coset $\Gamma_{i}\varphi$ of the stabilizer 
$\Gamma_{i} = \Gamma_{\{F_i\}} = \langle R_k\mid k\neq i\rangle$ of the base $i$-face $F_{i}$ in $\Gamma$.
\smallskip

We conclude this section with a remark about the flag stabilizers of arbitrary regular complexes. A priori only little can be said about their structure. However, there are bounds on the prime divisors of the group order. For a regular complex with a finite flag stabilizer $R_{-1}=\Gamma_\Phi$, the prime divisors of the order of $R_{-1}$ are bounded by 
\[{\rm max}(k_{i}\!-\!1\mid 0\leq i\leq n-1).\]
In fact, an element of $R_{-1}$ of prime order exceeding this number would necessarily have to fix all adjacent flags of a flag that it fixes. But a simple flag connectivity argument shows that in a regular complex only the trivial automorphism can have this property.

\section{Regular complexes from groups}
\label{comfgroups} 

In the previous section we derived various properties of flag-transitive subgroups of the automorphism groups of regular complexes. In particular, in Lemma~\ref{incchar}, we proved that the combinatorial structure of a regular $n$-complex $\K$ can be completely described in terms of the distinguished generating subgroups $R_{-1},R_{0},\ldots,R_{n}$ of any flag-transitive subgroup $\Gamma$ of~$\Gamma(\K)$. 

If $\K$ is a regular $n$-polytope, then $\Gamma(\K)$ is simply flag-transitive and hence has no proper flag-transitive subgroup.  In this case $\Gamma=\Gamma(\K)$, the base flag stabilizer $R_{-1}=R_{n}=\Gamma_\Phi$ is trivial, and each subgroup $R_i$ (with $i=0,\ldots,n-1$) is generated by an involutory automorphism $\rho_i$ which maps the base flag $\Phi$ to its unique $i$-adjacent flag. The group of $\K$ is then what is called a {\em string C-group\/} (see \cite[Ch. 2E]{arp}), that is, the {\it distinguished generators\/} $\rho_0, \ldots, \rho_{n-1}$ satisfy both the commutativity relations typical of a Coxeter group with a string diagram, and the intersection property of Lemma~\ref{intprop}, which now takes the form
\begin{equation}
\label{polyintprop}
\langle\rho_{i}\,|\,i\in I\rangle \cap \langle\rho_{i}\,|\,i\in J\rangle 
= \langle\rho_{i}\,|\,i\in I\cap J\rangle 
\qquad (I,J\subseteq \{0,\ldots,n-1\}).
\end{equation}

In this section we characterize the groups that can occur as flag-transitive subgroups of the automorphism group of a regular complex as what we will call here generalized string C-groups (with trivial core).  As one of the most important consequences of this approach, we may think of regular complexes and corresponding generalized string C-groups as being essentially the same objects. We follow \cite[Sect. 3]{kom2} (and \cite{esdiss}).

Let $\Gamma$ be a group generated by subgroups $R_{-1},R_{0},\ldots,R_{n}$, where $R_{-1}$ and $R_n$ are proper subgroups of $R_i$ for each $i=0,\ldots,n-1$, and $R_{n}=R_{-1}$; we usually assume that $n \geq 1$. These subgroups are the {\em distinguished generating subgroups\/} of $\Gamma$, and along with $\Gamma$ will be kept  fixed during this section.  As before we set $N := \{-1,0,\ldots,n\}$. Further, the subgroups $\Gamma_{I} := \langle R_i \mid i\in I\rangle$ with $I \subseteq N$ are called the {\em distinguished subgroups\/} of $\Gamma$; here $\Gamma_{\emptyset} = R_{-1}$. Then $\Gamma$ is called a {\em generalized C-group\/} if $\Gamma$ has the following {\em intersection property (with respect to its distinguished generating subgroups)\/}: 
\begin{equation}
\label{cgr}
\Gamma_{I} \cap \Gamma_{J} = \Gamma_{I \cap J} \quad (I,J\subseteq N).
\end{equation}
It is immediate from the definition that the distinguished subgroups $\Gamma_I$ are themselves
generalized C-groups, with distinguished generating subgroups those $R_{i}$ with $i \in I\cup \{-1,n\}$.

It also follows from the definition that, in a generalized C-group $\Gamma$, the subgroups $\Gamma_I$ with $I\subseteq \{0,\ldots,n-1\}$ are pairwise distinct.  To see this, first observe that by (\ref{cgr}) and our assumption that $R_{-1}\,(=\Gamma_\emptyset)$ be a proper subgroup of $R_i$ for each $i=0,\ldots,n-1$, a group $R_{i}$ cannot be a subgroup of a group $\Gamma_{I}$ when $i \notin I\cup \{-1,n\}$. Consequently, if $I,J \subseteq \{0,\ldots,n-1\}$ and $\Gamma_{I} = \Gamma_{J}$, then (\ref{cgr}) implies that $\Gamma_{I} = \Gamma_{I \cap J} = \Gamma_J$; hence it follows from what was said before that $I = I \cap J = J$, as required.

A generalized C-group $\Gamma$ is called a {\em generalized string C-group\/} (strictly speaking, a {\it string generalized C-group\/}) if its generating subgroups satisfy 
\begin{equation}
\label{scgr}
R_{i}R_{j}=R_{j}R_{i}  \qquad (-1\leq i <j-1 \leq n-1).
\end{equation}

Thus, for the remainder of this section we assume that $\Gamma = \langle R_{-1},R_{0},\ldots,R_{n}\rangle$, with $n \geq 1$, is a generalized string C-group. 

As in the previous section, for each $i\in N$ we write 
\begin{equation}
\begin{array}{lllll}
\Gamma_{i}   &:=& \langle R_j \mid j\neq i\rangle, \\[.03in]
\Gamma_{i}^{-} &:=& \langle R_j \mid j\leq i\rangle, \\[.03in]
\Gamma_{i}^{+}&:= & \langle R_j \mid j\geq i\rangle, 
\end{array} 
\end{equation}
so in particular, $\Gamma_{-1}=\Gamma_{n}=\Gamma$. Then (\ref{gamij}) and (\ref{gami}) carry over, as before by the commuting properties (\ref{scgr}). Observe that the subgroups $\Gamma_{0},\ldots,\Gamma_{n-1}$ are mutually distinct, and distinct from $\Gamma$.

We now construct a regular $n$-incidence complex $\K$ from $\Gamma$. For $i \in N$, we take as the set of $i$-faces of $\K$ (that is, its faces of rank $i$) the set of all right cosets $\Gamma_{i}\varphi$ in $\Gamma$, with $\varphi \in\Gamma$.  As improper faces of $\K$, we choose two copies of $\Gamma$, one denoted
by $\Gamma_{-1}$, and the other by $\Gamma_{n}$;  in this context, they are regarded
as distinct. Then, for the right cosets of $\Gamma_{-1}$ and $\Gamma_{n}$, we have $\Gamma_{-1}\varphi = \Gamma_{-1}$ and $\Gamma_{n}\varphi = \Gamma_{n}$ for all $\varphi \in \Gamma$.  On (the set of all proper and improper faces of) $\K$, we define the following partial order:
\begin{equation}
\label{incdef}
\Gamma_{i}\varphi \leq \Gamma_{j}\psi\; :\Longleftrightarrow\; -1 \leq i \leq j \leq n,\ \varphi\psi^{-1} \in
\Gamma_{i+1}^{+}\Gamma_{j-1}^{-}.
\end{equation}
Then $\Gamma$ acts on $\K$ in an obvious way as a group of order preserving automorphisms.

Alternatively the partial order on $\K$ can be defined by 
\begin{equation}
\label{coxcompinc}
\Gamma_{i}\varphi \leq \Gamma_{j}\psi\; :\Longleftrightarrow\; 
-1 \leq i \leq j \leq n,\ \Gamma_{i}\varphi \cap\Gamma_{j}\psi \neq \emptyset.
\end{equation}
The equivalence of the two definitions is based on the commutation rules of (\ref{scgr}). In fact, it follows as in equation (\ref{gami+1j-1}) that $\Gamma_{i+1}^{+}\Gamma_{j-1}^{-}=\Gamma_{i}\Gamma_{j}$, so 
$\varphi\psi^{-1} \in\Gamma_{i+1}^{+}\Gamma_{j-1}^{-}=\Gamma_{i}\Gamma_{j}$ if and only of 
$\Gamma_{i}\varphi \cap \Gamma_{j}\psi\neq\emptyset$.

If the dependence of $\K$ on $\Gamma$ and $R_{-1},R_{0},\ldots,R_{n}$ is to be emphasized, we write $\K(\Gamma)$ or $\K(\Gamma;R_{-1},R_{0},\ldots,R_{n})$ for $\K$.  

We first show that the condition (\ref{incdef}) induces a partial order on $\K$. For reflexivity and antisymmetry of $\leq$ we can appeal to (\ref{coxcompinc}).  Certainly, a coset $\Gamma_{i}\varphi$ is incident with itself, which is reflexivity.  If $\Gamma_{i}\varphi$ and $\Gamma_{j}\psi$ are two cosets with $\Gamma_{i}\varphi\leq\Gamma_{j}\psi$ and $\Gamma_{j}\psi\leq\Gamma_{i}\varphi$, then $i=j$ and the cosets (for the same subgroup) must coincide as they intersect; this implies antisymmetry.  Finally, if $-1 \leq i \leq j \leq k \leq n$, we have
\begin{equation}
\label{trans}
\Gamma_{j+1}^{+}\Gamma_{i-1}^{-} \cdot \Gamma_{i+1}^{+}\Gamma_{k-1}^{-}
= \Gamma_{j+1}^{+}\Gamma_{i+1}^{+}\Gamma_{i-1}^{-}\Gamma_{k-1}^{-}
= \Gamma_{i+1}^{+}\Gamma_{k-1}^{-}.
\end{equation}
Transitivity of $\leq$ then is an immediate consequence if we appeal to the original definition of $\leq$ in \eqref{incdef}.  Thus $\leq$ is a partial order.

Clearly, $\Phi := \{\Gamma_{-1},\Gamma_{0},\ldots,\Gamma_{n-1},\Gamma_n\}$ is a flag of $\K$, which we naturally call the {\em base\/} flag; its faces are also called the {\em base\/} faces of $\K$. Since $\Phi$ is a flag, so is its image 
$\Phi\varphi =\{\Gamma_{-1}\varphi,\Gamma_{0}\varphi,\ldots,\Gamma_{n-1}\varphi,\Gamma_n\varphi\}$
for each $\varphi \in \Gamma$. 

We next establish that $\Gamma$ acts transitively on all chains of $\K$ of each given type $I \subseteq N$. When $I=N$ this shows that $\Gamma$ acts transitively on the flags of $\K$. Now let $I\subseteq N$, and let 
$\{\Gamma_{i}\varphi_{i} \mid i \in I\}$ be a chain of type $I$.  We proceed by induction.  Suppose that, for some $k \in I$, we have already shown that there exists an element $\psi \in \Gamma$ such that $\Gamma_{i}\varphi_{i} = \Gamma_{i}\psi$ for each $i\in I$ with $i \geq k$.  Let $j \in I$ be the next smaller number than $k$ (assuming that there is one).  Then $\Gamma_{j}\varphi_{j} \leq \Gamma_{k}\psi$ implies by (\ref{incdef}) that
$\varphi_{j}\psi^{-1} \in \Gamma_{k+1}^{+}\Gamma_{j-1}^{-}$, say $\varphi_{j}\psi^{-1} = \alpha\beta$, with
$\alpha \in \Gamma_{k+1}^{+}$ and $\beta \in\Gamma_{j-1}^{-}$.  It follows that $\alpha^{-1}\varphi_{j} =
\beta\psi =: \chi$, say, and hence that
\[\begin{array}{lll}
\Gamma_{i}\chi =  \Gamma_{i}\beta\psi = \Gamma_{i}\psi, \; \mbox{for } i\in I,i \geq k,\\[.02in]
\Gamma_{j}\chi = \Gamma_{j}\alpha^{-1}\varphi_{j} = \Gamma_{j}\varphi_{j},
\end{array} \]
giving the same property with $j$ instead of $k$ (and $\psi$ replaced by $\chi$). This is the inductive step, and the transitivity follows.

If $I\subseteq N$ and $\Phi_I$ denotes the subchain of $\Phi$ of type $I$ (consisting of the faces in $\Phi$ with ranks in $I$), then the stabilizer of $\Phi_{I}$ in $\Gamma$ is the subgroup $\Gamma_{\bar{I}}$. In particular, 
the stabilizer of the base flag $\Phi$ itself is $R_{-1}$. In fact, an element $\varphi\in\Gamma$ stabilizes $\Phi_{I}$ if and only if $\Gamma_{i}\varphi = \Gamma_{i}$ for each $i \in I$.  Equivalently, $\Phi_{I}\varphi=\Phi_I$ if and only if 
\[  \varphi \,\in\, \bigcap_{i \in I} \,\Gamma_{i}  \,=\,\bigcap_{i \in I}\, \langle R_{j}
\mid j \neq i\rangle \,=\, \Gamma_{\bar{I}},  \]
by the intersection property (\ref{cgr}) for $\Gamma$.  Thus the stabilizer of $\Phi_I$ is $\Gamma_{\bar{I}}$.

We can now state the following theorem.
\begin{theorem}
\label{pofa}
Let $n \geq 1$, and let $\Gamma = \langle R_{-1},R_{0},\ldots,R_{n}\rangle$ be a generalized string C-group and $\K := \K(\Gamma)$ the corresponding partially ordered set. Then $\K$ is a regular $n$-complex on which $\Gamma$ acts flag-transitively. In particular, $\K$ is finite, if $\Gamma$ is finite.
\end{theorem} 

\begin{proof}
For $\K$ we need to check the defining properties (I1), \ldots, (I4) of incidence complexes.  The property (I1) is trivially satisfied with $\Gamma_{-1}$ and $\Gamma_{n}$ as the least and greatest face, respectively.  In fact, by (\ref{incdef}), $\Gamma_{-1} \leq \Gamma_{i}\varphi \leq\Gamma_{n}$ for all $\varphi$ and all $i$.

Next, we exploit the fact that every chain $\Omega$ in $\K$ of type $I$ can be expressed in the form 
$\Omega=\Phi_{I}\varphi$, for some $\varphi\in \Gamma$. In particular, $\Omega$ is contained in the flag $\Phi\varphi$, which gives (I2).

We then prove (I4). Now if we take $I = N \setminus\{i\}$ for any $i \in \{0,\ldots,n-1\}$, we see that the stabilizer of $\Phi_{N \setminus\{i\}} = \{\Gamma_{-1},\Gamma_{0},\ldots,\Gamma_{i-1},\Gamma_{i+1},\ldots,\Gamma_{n}\}$ is $\Gamma_{\{i\}} = R_{i}$. On the other hand, the stabilizer of $\Phi$ itself is $\Gamma_{\emptyset}=R_{-1}$, which by assumption is a proper subgroup of $R_{i}$. Hence the number of flags of $\K$ containing 
$\Phi_{N \setminus \{i\}}$, which is given by $|R_{i}:R_{-1}|$, is at least~$2$. The transitivity of $\Gamma$ on 
chains of type $N \setminus\{i\}$ then gives (I4). Thus, for each flag, the number of $i$-adjacent flags is at least $1$ and is given by $(|R_{i}:R_{-1}|-1)$ for each $i=0,\ldots,n-1$.

Finally, we demonstrate (I3), in the alternative form (I3') of strong flag-connectedness. As $\Gamma$ acts flag-transitively on $\K$, to prove (I3'), it suffices to consider the special case where one flag is the base flag $\Phi$.  If $\Psi$ is another flag of $\K$, let $J\subseteq N$ be such that $\Phi \cap \Psi = \Phi_{J}$. 
Since $\Phi_{J} \subseteq \Psi$ and the stabilizer of $\Phi_{J}$ is $\Gamma_{\bar{J}}$, the flag-transitivity of $\Gamma$ shows that $\Psi = \Phi\varphi$ for some $\varphi \in\Gamma_{\bar{J}}$. Suppose 
$\varphi = \varphi_{1}\ldots\varphi_{k}$ such that $\varphi_{l}\in R_{j_l}$, for some $j_{l}\in \overline{J}$, for $l=1,\ldots,k$. Define $\psi_{l} := \varphi_{l}\ldots \varphi_{k}$ for $l=1,\ldots,k$. Then 
\[ \Phi, \Phi \psi_{k},\Phi \psi_{k-1},\ldots, \Phi \psi_{2},
\Phi \psi_{1}=\Phi\varphi=\Psi ,\]
is a sequence of successively adjacent flags, all containing $\Phi_{J}$, which connects $\Phi$ and $\Psi$. Note here that $\Phi\psi_{l+1}$ and $\Phi\psi_{l}$ are $j_l$-adjacent for each $l=1,\ldots,k-1$, since $\Phi$ and $\Phi\varphi_l$ are $j_{l}$-adjacent and so are $\Phi\psi_{l+1}$ and $\Phi\varphi_l\psi_{l+1}=\Phi\psi_{l}$. Thus $\K$ is strongly connected, and the proof of the theorem is complete.
\end{proof}

Note that the action of $\Gamma$ on $\K:=\K(\Gamma)$ need not be faithful in general. The kernel of the action consists of the elements of $\Gamma$ which act trivially on $\K$, or equivalently, on the set of flags of~$\K$. The stabilizer of the base flag $\Phi$ in $\Gamma$ is $R_{-1}$, and hence the stabilizer of a flag $\Phi\varphi$ with $\varphi\in\Gamma$ is $\varphi^{-1}R_{-1}\varphi$. Hence the kernel of the action of $\Gamma$ of $\K$ is given by its subgroup
\begin{equation}
\label{core}
{\rm core}(R_{-1}) \,=\, \bigcap_{\varphi\in\Gamma}\, (\varphi^{-1}R_{-1}\varphi). 
\end{equation}
Recall that in a group $B$, the {\it core\/} of a subgroup $A$, denoted ${\rm core}(A)$, is the largest normal subgroup of $B$ contained in $A$; that is, ${\rm core}(A) = \cap_{b\in B}b^{-1}Ab$ (see~\cite{asch}). Clearly, $\Gamma$ itself can be identified with a flag-transitive subgroup of the automorphism group of $\K(\Gamma)$ if and only if ${\rm core}(R_{-1})$ is trivial. 

Our next theorem describes the structure of the sections of the regular complex $\K(\Gamma)$. For the proof we require the following consequence of the intersection property (\ref{cgr}):
\begin{equation}
\label{isomsec}
\Gamma_{k+1}^{+}\Gamma_{l-1}^{-} \,\cap\, \Gamma_{\{i+1,\ldots,j-1\}}
\,=\, \Gamma_{\{k+1,\ldots,j-1\}}\Gamma_{\{i+1,\ldots,l-1\}}\quad (i\leq k\leq l \leq j) 
\end{equation}
To prove this property, suppose $\varphi$ is an element in the set on the left hand side, $\varphi=\alpha\beta$ (say), with
$\alpha \in \Gamma_{k+1}^{+}$ and $\beta \in\Gamma_{l-1}^{-}$. Now apply (\ref{cgr}) twice, bearing in mind that $i\leq k\leq l \leq j$:\ first, with $I=\{-1,0,\ldots,l-1\}$ and $J=\{i+1,\ldots,n\}$ to obtain
\[ \beta =\alpha^{-1}\varphi\in \Gamma_{l-1}^{-} \cap \Gamma_{i+1}^{+} 
= \Gamma_{\{i+1,\ldots,l-1\}}, \]
and second, with $I=\{k+1,\ldots,n\}$ and $J=\{-1,0,\ldots,j-1\}$ to obtain
\[ \alpha = \varphi\beta^{-1} \in \Gamma_{k+1}^{+} \cap \Gamma_{j-1}^{-}
= \Gamma_{\{k+1,\ldots,j-1\}}. \]
Thus $\varphi\in\Gamma_{\{k+1,\ldots,j-1\}}\Gamma_{\{i+1,\ldots,l-1\}}$, as required. The opposite inclusion is clear, since
the two groups $\Gamma_{\{k+1,\ldots,j-1\}}$ and $\Gamma_{\{i+1,\ldots,l-1\}}$ both lie in $\Gamma_{\{i+1,\ldots,j-1\}}$, and are subgroups of $\Gamma_{k+1}^{+}$ and $\Gamma_{l-1}^{-}$ respectively. 

\begin{theorem}
\label{cgrpolprop}
Let $\K:=\K(\Gamma)$ be the regular $n$-complex associated with the generalized string C-group $\Gamma
= \langle R_{-1},R_{0},\ldots,R_{n}\rangle$.\\[.03in]
(a)\ Let $-1 \leq i< j-1 \leq n-1$, and let $F$ be an $i$-face and $G$ a $j$-face of $\K$ with $F \leq G$. Then the section $G/F$ of $\K$ is isomorphic to 
\[ \K(\Gamma_{\{i+1,\ldots,j-1\}})=\K(\Gamma_{\{i+1,\ldots,j-1\}};R_{-1},R_{i+1},\ldots,R_{j-1},R_{n}). \]
(b)\ The facets and vertex-figures of $\K$ are isomorphic to the regular $(n-1)$-complexes $\K(\Gamma_{n-1})$ and $\K(\Gamma_0)$, respectively.\\[.03in]
(c)\ Let $-1 \leq i \leq n-1$, and let $F$ be an $(i-1)$-face and $G$ an $(i+1)$-face of $\K$ with $F \leq G$. Then the number of $i$-faces of $\K$ in $G/F$ is $|R_{i}:R_{-1}|$.
\end{theorem}

\begin{proof}
We already established part (c). Part (b) is a special case of part (a). Now for part (a) assume that $-1 \leq i < j-1 \leq n-1$. The transitivity of $\Gamma$ on chains of type $\{i,j\}$ implies that it suffices to prove the result for the section $\K(i,j) := \Gamma_{j}/\Gamma_{i}$ of $\K$.  Let $I :=\{0,\ldots,i,j,\ldots,n-1\}$. There is a one-to-one correspondence between chains of $\K(i,j)$ and chains of $\K$ which contain $\Phi_{I}$. In particular, appealing again to the transitivity of $\Gamma$ on chains of a given type, we deduce that each face $\Gamma_{k}\varphi \in \K(i,j)$ (with $i\leq k \leq j$) admits a representation with $\varphi$ in the stabilizer of $\Phi_{I}$, namely $\Gamma_{\bar{I}}=\Gamma_{\{i+1,\ldots,j-1\}}$. It now follows from (\ref{isomsec}) that $\K(i,j)$ is isomorphic to $\K(\Gamma_{\{i+1,\ldots,j-1\}})$. Set $\Lambda:=\Gamma_{\{i+1,\ldots,j-1\}}$. In fact, by (\ref{isomsec}), if $\varphi,\psi\in\Gamma_{\{i+1,\ldots,j-1\}}=\Lambda$ and $i\leq k\leq l \leq j$, then $\varphi\psi^{-1}\in\Gamma_{k+1}^{+}\Gamma_{l-1}^{-}$ if and only if 
\[\varphi\psi^{-1}\,\in\,
\Gamma_{\{k+1,\ldots,j-1\}}\Gamma_{\{i+1,\ldots,l-1\}}
\,=\, \Lambda_{k+1}^{+}\,\Lambda_{l-1}^{-}, \]
or equivalently, $\Gamma_{k}\varphi\leq\Gamma_{l}\psi$ in $\K(i,j)$ (that is, in $\K$) if and only if 
\[\Lambda_{k}\,\varphi\leq\Lambda_{l}\,\psi\] 
in $\K(\Lambda)=\K(\Gamma_{\{i+1,\ldots,j-1\}})$. Thus $\K(i,j)$ and $\K(\Gamma_{\{i+1,\ldots,j-1\}})$ are isomorphic complexes, and the proof of the theorem is complete.
\end{proof}

There are two immediate consequences of the earlier results and the construction of $\K(\Gamma)$.

\begin{corollary}
\label{corgscg}
The generalized string C-groups with trivial core are precisely the flag-transitive subgroups of the automorphism groups of regular complexes.
\end{corollary}

\begin{proof}
Clearly, a flag-transitive subgroup of the automorphism group of a regular complex must have trivial core since only the identity automorphism fixes every face. This shows one direction. The converse was already addressed above.
\end{proof}

\begin{theorem} 
\label{pap}
Let $n \geq 1$, let $\K$ be a regular $n$-complex, let $\Gamma$ be a flag-transitive subgroup of $\Gamma(\K)$,
and let $R_{-1},R_{0},\ldots,R_{n}$ be the distinguished generating subgroups of $\Gamma$ associated with the base flag $\Phi = \{F_{-1},F_{0},\ldots,F_{n}\}$ of $\K$. Then the regular complexes $\K$ and $\K(\Gamma)$ (or more exactly, $\K(\Gamma;R_{-1},R_{0},\ldots,R_{n})$) are isomorphic.  In particular, the mapping 
$\K \to \K(\Gamma)$ given by
\[  F_{i}\varphi \rightarrow \Gamma_{i}\varphi \qquad  (-1 \leq i \leq n;\,\varphi \in \Gamma) \]
is an isomorphism.
\end{theorem}

\section{Regular polytopes and C-groups}
\label{regpols}

The basic structure results for abstract regular polytopes and their automorphism groups can be derived from the results of Sections~\ref{groupfcom} and \ref{comfgroups} (see \cite{esdiss,kom2} and \cite{arp}). Abstract polytopes are incidence complexes in which every flag has exactly one $i$-adjacent flag for each $i=0,\ldots,n-1$; that is, polytopes satisfy the diamond condition. Regular polytopes have a simply flag-transitive automorphism group, so all flag stabilizers are trivial and there are no proper flag-transitive subgroups.

Let $\K$ be a regular $n$-polytope, $\Phi=\{F_{-1},F_{0},\ldots,F_{n}\}$ be a base flag of $\K$, and let $\Gamma:=\Gamma(\K)$. Then $R_{i}=\langle \rho_i\rangle$ for each $i=0,\ldots,n-1$, where $\rho_i$ is the unique automorphism of $\K$ mapping $\Phi$ to its $i$-adjacent flag. The subgroups $R_{-1}$ and $R_n$ are trivial. Thus $\Gamma = \langle\rho_0,\ldots,\rho_{n-1}\rangle$. The involutions $\rho_{0},\ldots,\rho_{n-1}$ are called the {\it distinguished generators\/} of $\Gamma$. Then the structure results of Section~\ref{groupfcom} for the distinguished generating subgroups of $\Gamma$ translate directly into corresponding statements for the distinguished generators.

Conversely, let $\Gamma$ be a group generated by involutions $\rho_0,\ldots,\rho_{n-1}$, called the {\it distinguished generators\/} of $\Gamma$. Then $\Gamma$ is a group of the type discussed at the beginning of Section~\ref{comfgroups}, with $R_{i}:=\langle\rho_{i}\rangle$ for $i=0,\ldots,n-1$, and $R_{-1}=R_{n}=\{1\}$. In this case it suffices to consider the distinguished subgroups $\Gamma_{I} := \langle \rho_i \mid i\in I\rangle$ with $I \subseteq \{0,\ldots,n-1\}$. We call $\Gamma$ a {\em C-group\/} if $\Gamma$ has the intersection property (\ref{cgr}), that is, 
\begin{equation}
\label{cgrrho}
\langle\rho_{i}\,|\,i\in I\rangle \,\cap\, \langle\rho_{i}\,|\,i\in J\rangle
\,=\, \langle\rho_{i}\,|\,i\in I\cap J\rangle \qquad (I,J\subseteq \{0,\ldots,n-1\}).
\end{equation}
A C-group $\Gamma$ is called a {\em string C-group\/} if the distinguished generators also satisfy the relations
\begin{equation}
\label{scgrrho}
(\rho_{i}\rho_{j})^{2}=1  \qquad (-1\leq i <j-1 \leq n-1),
\end{equation}
which is equivalent to requiring (\ref{scgr}). The number of generators $n$ is called the {\it C-rank\/}, or simply the {\it rank\/}, of $\Gamma$. Clearly, C-groups are generalized C-groups, and string C-groups are generalized string C-groups. 

The regular $n$-complex $\K=\K(\Gamma)$ of Theorem~\ref{pofa} associated with a string C-group $\Gamma=\langle\rho_{0},\ldots,\rho_{n-1}\rangle$ is a polytope, by Theorem~\ref{cgrpolprop}(c). Thus $\K$ is a regular $n$-polytope, with partial order given by (\ref{incdef}), or equivalently, (\ref{coxcompinc}). The relevant subgroups involved in describing the partial order are $\Gamma_{i} = \langle \rho_j \mid j\neq i\rangle$, $\Gamma_{i}^{-} = \langle \rho_j \mid j\leq i\rangle$, and $\Gamma_{i}^{+} = \langle \rho_j \mid j\geq i\rangle$. The $i$-faces of $\K$ are the right cosets of $\Gamma_{i}$ for each $i$. 

Thus the string C-groups are precisely the groups of regular polytopes.

Abstract polytopes of rank $3$ are also called ({\it abstract\/}) {\it polyhedra\/}. Regular polyhedra are regular maps on surfaces, and most regular maps on surfaces are regular polyhedra (see \cite{con,cm}).

A regular $n$-polytope $\K$ is of ({\it Schl\"afli\/}) {\it type\/} $\{p_{1},\ldots,p_{n-1}\}$ if its sections $G/F$ of rank $2$ defined by an $(i-2)$-face $F$ and an $(i+1)$-face $G$ with $F<G$ are isomorphic to $p_{i}$-gons (possibly $p_{i}=\infty$) for $i=1,\ldots,n-1$; then $p_{i}$ is the order of $\rho_{i-1}\rho_i$ in $\Gamma(\K)$. 

Coxeter groups are a particularly important class of C-groups (see \cite[Ch. 3]{arp} and \cite{hu,esdiss,kom2}). Let $p_{1},\ldots,p_{n-1}\geq 2$, and let $\Gamma=\langle\rho_{0},\ldots,\rho_{n-1}\rangle$ be the (string) Coxeter group defined by the relations
\begin{equation}
\label{coxrels}
\begin{array}{lcl}
\rho_i^2                      &   = 1  &  \mbox{for  } 0 \leq i \leq n-1 ; \\
{(\rho_i\rho_j)}^2         &  =1  &  \mbox{for  } 0 \leq i < j-1\leq n-2 ; \\
{(\rho_{i-1}\rho_i)}^{p_i} & = 1  &  \mbox{for  } 1 \leq i \leq n-1.
\end{array}
\end{equation}
Then $\Gamma$ is a string C-group. The corresponding regular $n$-polytope is called the {\it universal regular polytope of type $\{p_{1},\ldots,p_{n-1}\}$\/} and is denoted by the Schl\"afli symbol $\{p_{1},\ldots,p_{n-1}\}$. This polytope covers every regular polytope of type $\{p_{1},\ldots,p_{n-1}\}$. For combinatorial and geometric constructions of the universal regular polytopes from the Coxeter complexes of the underlying Coxeter groups $\Gamma$ see \cite[Sect. 3D]{arp} (or \cite{esdiss,kom2}).

The regular convex polytopes and regular tessellations (or honeycombs) of spherical, Euclidean or hyperbolic spaces are particular instances of universal regular polytopes, with the type determined by the standard 
Schl\"afli symbol.   

\section{Extensions of regular complexes}
\label{ext}

A central problem in the classical theory of regular polytopes is the construction of polytopes with prescribed facets. In this section, we briefly investigate the corresponding problem for regular incidence complexes. We say that a regular complex $\mathcal{L}$ is an {\it extension\/} of a regular complex $\K$ if the facets of $\mathcal{L}$ are isomorphic to $\K$ and if all automorphisms of $\K$ are extended to automorphisms of $\CL$. In conjunction with the former condition the latter condition means that the stabilizer of a facet of $\CL$ in $\Gamma(\CL)$ contains $\Gamma(\K)$ as a subgroup; or, less formally, $\Gamma(\K)\leq\Gamma(\CL)$.

Let $\K$ be a regular $n$-complex with $n\geq 1$, let $\Phi=\{F_{-1},F_{0},\ldots,F_{n}\}$ be a base flag of~$\K$, let $\Gamma$ be a flag-transitive subgroup of $\Gamma(\K)$, and let $R_{-1},R_{0},\ldots,R_{n}$ be the distinguished generating subgroups of $\Gamma$ associated with $\Phi$. In constructing extensions of $\K$ we consider certain groups $\Lambda$ with distinguished systems of generating subgroups $R'_{-1},R'_0,\ldots,R'_{n+1}$. We use similar notation for the distinguished subgroups of $\Lambda$ as for $\Gamma$. The following theorem was proved in \cite{esdiss,kom3} (see also \cite{onarr}).

\begin{theorem}
\label{extthm}
Let $\K$ be a regular $n$-complex with $n\geq 1$, and let $\Gamma = \langle R_{-1},R_{0},\ldots,R_{n}\rangle$ be a flag-transitive subgroup of $\Gamma(\K)$, as above. Let $\Lambda$ be a group generated by a system of subgroups $R'_{-1},R'_0,\ldots,R'_{n+1}$ satisfying the following conditions (a), (b) and (c).\\[.04in]
(a)\; $R'_{-1}=R'_{n+1}\subset R'_n$, $\Lambda\neq\Lambda_{n-1}^{-}$;\\[.02in]
(b)\; $R'_{i}R'_{j}=R'_{j}R'_{i}$ for $0\leq i<j-1\leq n-1$; \\[.02in]
(c)\; There exists a surjective homomorphism $\pi: \Lambda_{n-1}^{-}\rightarrow \Gamma$ such that \\[.02in]
\indent \; (c1)\, $\pi^{-1}(R_{i})=R'_i$ for $i=-1,0,\ldots,n-1$;        \\[.02in]
\indent \; (c2)\, $\Lambda_{i}^{+}\cap \Lambda_{n-1}^{-} =\pi^{-1}(\Gamma_{i}^{+})$ for $i=-1,0,\ldots,n$. \\[.04in]
Then there exists a regular $(n+1)$-complex $\CL$ with facets isomorphic to $\K$. In particular, $\Lambda$~acts  flag-transitively on $\CL$, and $\CL$ is finite if $\Lambda$ is finite. If $\pi$ is an isomorphism, then $\Lambda$ is isomorphic to a flag-transitive subgroup of $\Gamma(\CL)$, the group $\Gamma$ is a subgroup of $\Lambda$, and $\CL$ is finite if and only if $\Lambda$ is finite. Further,~$\CL$ is a lattice, if $\K$ is a lattice and $\Lambda$ satisfies the following condition:\\[.04in]
(d)\; Let $0\leq i\leq j<k\leq n$ and $\tau\in \Lambda_{k-1}^{-}$. If 
$\tau\not\in \Lambda_{i+1}^{+}\Lambda_{j-1}^{-}$ and if 
$\tau\not\in \Lambda_{i+1}^{+}\Lambda_{l-1}^{-}\Lambda_{\{j+1,\ldots,k-1\}}$ for each $l$ with $j<l<k$, then
$\Lambda_{j+1}^{+}\cap \Lambda_{n-1}^{-}\Lambda_{i+1}^{+}\tau \subseteq 
\Lambda_{n-1}^{-}\Lambda_{k+1}^{+}$.
\end{theorem}

Note that condition (d) of Theorem~\ref{extthm} can be reformulated as follows: 
if $0\leq i\leq j<k\leq n$, $\tau\in \Lambda_{n-1}^{-}$, and $F_k$ is the supremum of $F_j$ and $F_{i}(\tau\pi)$ in $\K$, then $\Lambda_{j+1}^{+}\cap \Lambda_{n-1}^{-}\Lambda_{i+1}^{+}\tau \subseteq 
\Lambda_{n-1}^{-}\Lambda_{k+1}^{+}$.

Theorem~\ref{extthm} translates the problem of finding an extension of a regular complex $\K$ into an embedding problem for its automorphism group $\Gamma(\K)$ into a suitable group $\Lambda$. A regular complex has many possible extensions. However, it is much harder to find an extension which is a lattice, if $\K$ is a lattice. 

The following result was proved in \cite{esdiss,kom3} (see also \cite{inctor}).

\begin{theorem}
\label{symgext}
Let $\K$ be a finite regular $n$-complex, and let $f$ denote the number of facets of~$\K$. Then $\K$ admits an extension $\CL$ whose automorphism group $\Gamma(\CL)$ contains a flag-transitive subgroup isomorphic to the symmetric group $S_{f+1}$. If $\K$ is a polytope, then $\CL$ is a polytope and $\Gamma({\CL})=S_{f+1}$.
\end{theorem}

In the extension $\CL$ of Theorem~\ref{symgext} the $(n-1)$-faces always lie in exactly two facets, regardless of whether or not $\CL$ is a polytope. For lattices $\K$, this complex $\CL$ is almost always again a lattice. A slightly modified construction for the group $\Lambda$, with $S_{f+1}$ replaced by the larger group $S_{f+1}\times \Gamma(\K)$, always guarantees that the corresponding extension of $\K$ is again a lattice if $\K$ is a lattice (see \cite{esext,inctor}). For an extension of a regular complex $\K$ it usually is the lattice property which is the hardest to verify.

For regular polytopes, further extension results have been obtained in recent years (for example, see \cite{pelext}). For these results, both $\K$ and $\CL$ are regular polytopes. There are also good extension results for chiral polytopes (see \cite{cunpel,swc,swfec}) and for hypertopes (see \cite{flw}).

There are also interesting infinite extensions $\CL$ of regular $n$-complexes $\K$, which in the case of polytopes have certain universality properties. Let $k\geq 2$ be an integer, and let $C_k$ denote the cyclic group of order $k$. Let $\Lambda$ be the amalgamated free product of $\Gamma:=\Gamma(\K)=\langle R_{-1},R_{0},\ldots,R_{n}\rangle$ and the direct product $\Gamma_{n-2}^{-}\times C_{k}$, with amalgamation of the subgroups  $\Gamma_{n-2}^{-}$ and $\Gamma_{n-2}^{-}\times \{1\}$ under the isomorphism $\kappa:\,\Gamma_{n-2}^{-}\rightarrow\Gamma_{n-2}^{-}\times \{1\}$ defined by $\varphi\rightarrow (\varphi,1)$. (For amalgamated free products see ~\cite{lynsch}). Then $\Lambda$ is the quotient of the free product of the two groups $\Gamma$ and $\Gamma_{n-2}^{-}\times C_{k}$ obtained by imposing on the free product the set of new relations 
\[ (\varphi)\kappa = \varphi \quad (\varphi\in \Gamma_{n-2}^{-}),\]
which in effect identify $\varphi$ and $(\varphi,1)$ for each $\varphi\in \Gamma_{n-2}^{-}$. Thus, in standard notation for amalgamated free products, 
\[\begin{array}{lclcl} 
\Lambda &\!\!=\!\!&\Gamma \!\!\!\!& * &\!\!\!\!(\Gamma_{n-2}^{-}\times C_{k}). \\[-.05in]
               &  &               & {}^{\Gamma_{n-2}^{-}} &
               \end{array}\]
We use slightly simpler notation and write 
\begin{equation}
\label{amalext}
\Lambda \,=\, \Gamma\, *_{\kappa}(\Gamma_{n-2}^{-}\times C_{k}).
\end{equation}
Then, with the distinguished generating system $R'_{-1},R'_0,\ldots,R'_{n+1}$ given by $R'_{i}:=R_{i}$ for $i\leq n-1$, $R'_{n}:=R_{-1}\times C_k$, and $R'_{n+1}:=R_{-1}$, this group $\Lambda$ turns out to be a generalized C-group. More explicitly, we have the following result (see~\cite{onarr}).

\begin{theorem}
\label{infexten}
Let $\K$ be a regular $n$-complex, and let $k\geq 2$. Then $\K$ admits an infinite extension $\CL$ whose automorphism group $\Gamma(\CL)$ contains a flag-transitive subgroup isomorphic to
$\Gamma\, *_{\kappa}(\Gamma_{n-2}^{-}\times C_{k})$. In $\CL$, each $(n-1)$-face lies in exactly $k$ facets 
(the $k$ facets containing the base $(n-1)$-face of $\CL$ are cyclically permuted by the subgroup $C_k$ of $R'_n$). Moreover, $\CL$ is a lattice if $\K$ is a lattice.  
\end{theorem}

If $\K$ is a regular $n$-polytope and $k=2$, then $\CL$ has the following universality property:\ if $\CP$ is any regular $(n+1)$-polytope with facets isomorphic to $\K$, then $\CP$ is covered by $\CL$. The polytope $\CL$ was called the {\em universal extension\/} of $\K$ in \cite[Ch. 4D]{arp}. For example, if $\K$ is the triangle $\{3\}$, then 
$\CL$ is the regular hyperbolic tessellation $\{3,\infty\}$ by ideal triangles, whose automorphism group is the projective general linear group $\rm{PGL}_{2}(\mathbb{Z})$.

Recently there has been a lot of progress in the study of combinatorial coverings of arbitrary abstract polytopes (see \cite{ha,mosch,mopewi}). For example, in the paper~\cite{mosch}, with Monson, it was shown that every finite abstract polytope is a quotient of a regular polytope of the same rank; that is, every finite abstract polytope has a finite regular cover.

For open questions related to extensions of regular complexes see Problem~\ref{pocom} in the next section.

\section{Abstract Polytope Complexes}
\label{apcs}

An incidence complex $\K$ of rank $n$ is called an {\it $(n-1)$-polytope complex\/}, or simply a {\it polytope complex\/}, if all facets of $\K$ are abstract polytopes. If the rank $n$ is $3$ or $4$ respectively, we also use the term {\it polygon complex\/} or {\it polyhedron complex}.  

Let $\K$ be a regular polytope complex of rank $n$, and let $\Phi:=\{F_{-1},F_{0},\ldots,F_n\}$ be a base flag of $\K$. Let $\Gamma=\langle R_{-1},R_{0},\ldots,R_{n}\rangle$ be a flag-transitive subgroup of $\Gamma(\K)$,  
where as before $R_{-1},R_{0},\ldots,R_{n}$ are the distinguished generating subgroups of $\Gamma$. Recall that  $R_{-1}=R_{n}=\Gamma_\Phi$, which is the stabilizer of $\Phi$ in $\Gamma$. Each subgroup $R_i$ with $i=0,\ldots,n-1$ acts transitively on the $k_{i}=:k_{i}(\K)$ faces of $\K$ of rank $i$ in $F_{i+1}/F_{i-1}$, and by (\ref{kis}) we  know that $|R_{i}:R_{-1}|=k_{i}$. As $\K$ is a polytope complex, $k_{i}=2$ for $i\leq n-2$ and $k:=k_{n-1}\geq 2$. Thus $|R_{i}:R_{-1}|=2$ if $i\leq n-2$, so $R_{-1}$ is a normal subgroup of $R_{i}$ in this case, and $|R_{n-1}:R_{-1}|=k$. It follows that $R_{-1}$ is also a normal subgroup of 
$\Gamma_{n-1}=\langle R_{-1},R_{0},\ldots,R_{n-2}\rangle$, the stabilizer of $F_{n-1}$ in~$\Gamma$. 

Moreover, $\Gamma_{n-1}$ acts flag-transitively on the base facet $F_{n-1}/F_{-1}$ of $\K$, and its subgroup $R_{-1}$ is the stabilizer of the base flag $\{F_{-1},F_{0},\ldots,F_{n-1}\}$ of $F_{n-1}/F_{-1}$ in $\Gamma_{n-1}$. Now $F_{n-1}/F_{-1}$ is a polytope since $\K$ is a polytope complex, so $R_{-1}$ must be the kernel of the action of $\Gamma_{n-1}$ on $F_{n-1}/F_{-1}$; in fact, in a flag-transitive action on a polytope, if a group element stabilizes a flag then it stabilizes every flag and thus also every face. Also, again because the facet is a polytope, the flag-transitive subgroup of the automorphism group $\Gamma(F_{n-1}/F_{-1})$ of the facet $F_{n-1}/F_{-1}$ induced by $\Gamma_{n-1}$ must be $\Gamma(F_{n-1}/F_{-1})$ itself. Thus $\Gamma_{n-1}/R_{-1}$ is a string C-group and 
\begin{equation}
\label{pckernel}
\Gamma_{n-1}/R_{-1} \cong \Gamma(F_{n-1}/F_{-1}).
\end{equation}

The skeletons of abstract polytopes provide interesting examples of polytope complexes. Let $\CL$ be an abstract $m$-polytope, and let $n\leq m$. The {\em $(n-1)$-skeleton\/} of $\CL$, denoted $skel_{n-1}(\CL)$, is the $n$-complex with faces those of $\CL$ of rank at most $n-1$ or of rank $m$ (the $m$-face of $\CL$ becomes the $n$-face of $skel_{n-1}(\CL)$). Then $skel_{n-1}(\CL)$ is an $(n-1)$-polytope complex whose facets are the $(n-1)$-faces of $\CL$. For example, the 2-complex $skel_{1}(\CL)$ can be viewed as a graph often called the {\em edge-graph\/} of $\CL$. 

Now suppose $\CL$ is a regular $m$-polytope and $\Gamma(\CL) = \langle \rho_{0},\ldots,\rho_{m-1}\rangle$, where $\rho_{0},\ldots,\rho_{m-1}$ are the distinguished involutory generators (with respect to a base flag of $\CL$).  Then the $(n-1)$-skeleton $\K:=skel_{n-1}(\CL)$ is a regular polytope complex of rank $n$ admitting a flag-transitive (but not necessarily faithful) action by $\Gamma:=\Gamma(\CL)$. This action of $\Gamma$ on $\K$ is faithful if $\CL$ is a lattice (but weaker assumptions suffice); in fact, in this case $\Gamma(\CL)$ acts faithfully on the vertex set of $\CL$ and hence also on the set of faces of $\CL$ of rank smaller than $n$.  In any case, the distinguished generating subgroups $R_{-1},R_{0},\ldots,R_{n}$ of $\Gamma$ for its action on $\K$ are given by $R_{-1}=R_{n}:=\langle\rho_{n},\ldots,\rho_{m-1}\rangle$, 
\[R_{i}:=\langle\rho_{i},\rho_{n},\ldots,\rho_{m-1}\rangle \;\,(\cong \langle\rho_{i}\rangle\times \langle\rho_{n},\ldots,\rho_{m-1}\rangle) \quad (i\leq n-2),\] and 
$R_{n-1}:=\langle \rho_{n-1},\rho_{n},\ldots,\rho_{m-1}\rangle$. (Here we are in a slightly more general situation than discussed in Section~\ref{groupfcom}, in that $\Gamma$ may act on $\K$ with nontrivial kernel, that is, 
$\Gamma$ may not be a subgroup of $\Gamma(\K)$. However, the corresponding results of Section~\ref{groupfcom} carry over to this situation as well.)  In particular, $R_{n-1}$ must be a string C-group of rank $m-n+1$. Note that the parameter $k=k_{n-1}$ for $\K$ is given by the number of $(n-1)$-faces of $\CL$ that contain a given $(n-2)$-face of $\CL$ (or equivalently, by the number of vertices of the co-face of $\CL$ at an $(n-2)$-face of $\CL$). For example, if $m=n+1$ and $\CL$ is a regular polytope of Schl\"afli type $\{p_{1},\ldots,p_{n}\}$, then $k=p_{n}$ and $R_{n-1}$ is the dihedral group~$D_{p_n}$. 

There are a number of interesting open problems concerning the characterization of regular polytope complexes which are skeletons of regular polytopes of higher rank. 

\begin{problem}
\label{skelprob}
Let $\K$ be a regular polytope complex of rank $n$ with automorphism group $\Gamma(\K)=\langle R_{-1},R_{0},\ldots,R_{n}\rangle$ and base flag $\{F_{-1},F_0,\ldots,F_{n}\}$, and let $k:=k_{n-1}(\K)>2$. Suppose $R_{n-1}$ is isomorphic to a string C-group.\\[.01in] 
(a)\ Is $\K$ always the $(n-1)$-skeleton of a regular polytope? \\[.01in] 
(b)\ Is $\K$ the $(n-1)$-skeleton of a regular $(n+l-1)$-polytope if $R_{n-1}$ is isomorphic to a string C-group of rank $l$?\\[.01in] 
(c)\ Suppose $\K$ is a lattice. Is $\K$ always the $(n-1)$-skeleton of a regular polytope which is also a lattice?\\[.01in] 
(d)\ Let $\CL$ denote a regular polytope with $(n-1)$-skeleton $\K$. What can be said about the structure of $\CL$ in the interesting special cases when $R_{n-1}$ acts on $F_{n}/F_{n-2}$ as a dihedral group $D_k$, alternating group $A_k$, or symmetric group $S_k$? \\[.01in] 
(e)\ Under which conditions on $\K$ is there a unique regular polytope with $\K$ as its $(n-1)$-skeleton? 
\end{problem}

Similar questions can be asked when $\Gamma(\K)$ is replaced by a flag-transitive subgroup $\Gamma$. 

There are also a number of open problems for regular polytope complexes with preassigned type of (polytope) facets.

\begin{problem}
\label{pocom}
Let $\mathcal{F}$ be a regular $(n-1)$-polytope, and let $k>2$.\\
(a)\ Among the regular polytope complexes $\K$ of rank $n$ with facets isomorphic to $\mathcal{F}$ and $k_{n-1}(\K)=k$, when is there a ``universal" polytope complex covering all these polytope complexes?\\
(b)\ Among the regular polytope complexes $\K$ of rank $n$ with facets isomorphic to $\mathcal{F}$, with $k_{n-1}(\K)=k$, and with $R_{n-1}$ acting on $F_{n}/F_{n-2}$ as a cyclic group $C_k$, dihedral group $D_k$, alternating group $A_k$, or symmetric group $S_k$, is there a ``universal" polytope complex covering all these polytope complexes?\\
(c)\ Among the regular polytope complexes $\K$ of rank $n$ with facets isomorphic to $\mathcal{F}$ and $k_{n-1}(\K)=k$, to what extent can one preassign the transitive permutation action of $R_{n-1}$ on $F_{n}/F_{n-2}$? In other words, which permutation groups on $k$ elements arise as $R_{n-1}$?\\
\end{problem}

Note that when Theorem~\ref{infexten} is applied to a regular $(n-1)$-polytope $\mathcal{F}$, interesting examples of regular polytope complexes $\K$ of rank $n$ arise in which $R_{n-1}$ acts on $F_{n}/F_{n-2}$ as a cyclic group $C_k$. The underlying construction of these complexes from group amalgamations has a somewhat ``universal flavor"; however, it is not known if these polytope complexes actually are universal among all polytope complexes with $R_{n-1}$ acting as $C_k$, unless $k=2$. When $k=2$ the corresponding complex is indeed universal and is known as the universal (polytope) extension of $\mathcal{F}$ (see \cite[Ch. 4D]{arp}).

Another interesting problem concerns the existence of simply flag-transitive subgroups. 

\begin{problem}
\label{flsubgroups}
Describe conditions on a regular polytope complex $\K$ that guarantee that $\Gamma(\K)$ contains a simply flag-transitive subgroup.
\end{problem}

The final set of problems we describe here concerns geometric realizations of regular polytope complexes or more general incidence complexes in real Euclidean spaces or unitary complex spaces. It is known that the set of all Euclidean realizations of a given finite abstract regular polytope, if not empty, has the structure of a convex cone called the {\it realization cone\/} of the given polytope (see \cite[Ch. 5]{arp} or \cite{mr,mm}). 

\begin{problem}
\label{real}
Let $\K$ be a finite regular polytope complex of rank $n$ with automorphism group $\Gamma(\K)=\langle R_{-1},R_{0},\ldots,R_{n}\rangle$ and facets isomorphic to $\mathcal{F}$, and let $k:=k_{n-1}(\K)$. Describe the realization space of $\K$ (that is, the set of all realizations of $\K$ in a Euclidean space) in terms of the realization cone of $\mathcal{F}$ and the permutation action of $R_{n-1}$ on the $k$ facets of $\K$ containing the base $(n-2)$-face (that is, the action of $R_{n-1}$ on $F_{n}/F_{n-2}$).
\end{problem}

Our next problem deals with geometric realizations in a Euclidean space of a given dimension $d$. 

\begin{problem}
\label{realinaspace}
Classify all geometrically regular polytope complexes in a Euclidean space of dimension $d\geq 4$.
\end{problem}

The case $d=3$ was solved in \cite{pesch1,pesch2}, where all geometrically regular polygon complexes (of rank 3) in ordinary Euclidean $3$-space were classified. The classification is quite involved.

There may be a nice theory of unitary complex realizations for certain types of incidence complexes (including perhaps the duals of regular polytope complexes). The well-known regular complex $n$-polytopes in unitary complex $n$-space $\mathbb{C}^n$ are examples of incidence complexes which are realized as affine complex subspace configurations on which the (unitary) geometric symmetry group acts flag-transitively (see~\cite{cox1}). These structures could serve as a guide to develop a complex realization theory for more general kinds of incidence complexes. 
 
\section{Notes}
\label{notes}

{\bf (1)}\ 
{\it Regular incidence complexes\/} were introduced around 1977 by Ludwig Danzer as combinatorial generalizations of regular polytopes~\cite{crp}, regular complex polytopes~\cite{cox1}, and other highly ``regular" incidence structures. The notion built on Branko Gr\"unbaum's work on {\it regular polystromata} (see \cite{grgcd}). The first systematic study of incidence complexes from the discrete geometry perspective occurred in my doctoral dissertation~\cite{esdiss} (and the related publications~\cite{kom1,kom2,kom3,onarr}), at about the same time when the concept of diagram geometries was introduced by Buekenhout~\cite{bu1} to find geometric interpretations for the sporadic simple groups. At the time of the writing of my dissertation, I was not aware of Gr\"unbaum's paper~\cite{grgcd}, nor did I know about Buekenhout's work~\cite{bu1}. (These were times before Google!) I learnt about both papers in 1981. Starting with \cite{inctor,amal}, my own work focussed on the class of incidence complexes now called abstract polytopes.

{\bf (2)}\ Incidence complexes satisfying the diamond condition (I4P) were originally called {\it incidence polytopes\/} (see \cite{kom1,kom2}). During the writing of \cite{arp} the new name {\it abstract polytopes\/} was adopted in place of {\it incidence polytopes\/}, and the name ({\it string\/}) {\it C-groups\/} (`C' standing for `Coxeter') was coined for the type of groups that are automorphism groups of abstract regular polytopes. Also, some of the original terminology of \cite{kom1} was changed; for example, the term `rank' was used in place of `dimension' (the term `dimension' was reserved for geometric realizations of abstract regular polytopes~\cite[Ch. 5]{arp} and \cite{mr,mm}).

{\bf (3)}\ Incidence complexes which are lattices were originally called {\it non-degenerate\/} complexes, indicating a  main focus on lattices consistent with ordinary polytope theory. As abstract polytope theory developed, this distinction played less of a role, in part also because the lattice property did not translate into an elegant property for the automorphism group (see, for example, condition (d) in Theorem~\ref{extthm}). 

{\bf (4)}\ Danzer's original definition of an incidence complex used the original connectivity condition (I3). The equivalent condition (I3') for strong flag-connectedness was first introduced in \cite{esdiss}. 

{\bf (5)}\ Our current condition (I4) is weaker than Danzer's original defining condition, which required that there be numbers $k_{0},\ldots,k_{n-1}$ such that, for any $i=0,\ldots,n-1$ and for any $(i-1)$-face $F$ and $(i+1)$-face $G$ with $F<G$, there are exactly $k_{i}$ $i$-faces $H$ with $F<H<G$. For regular (and many other kinds of highly symmetric) incidence complexes the two conditions are equivalent (see equation (\ref{kis})).

{\bf (6)}\ In \cite{kom2}, the faces of the regular $n$-complex $\K(\Gamma)$ of Section~\ref{comfgroups} for a given group $\Gamma$ were denoted by formal symbols of the form $[\varphi,F_{i}]$, where $\varphi\in \Gamma$ and $F_{i}$ is from a fixed $(n+2)$-set (in a way, the base flag). This description of the faces is equivalent to the coset description of faces adopted above, with $[\varphi,F_{i}]$ corresponding to $\Gamma_{i}\varphi$ (see the footnote on p.40 of \cite{kom2}). In \cite{kom2}, there were no explicit analogues of the conditions of Lemma~\ref{incchar}(c) and  (\ref{coxcompinc}) both of which describe the partial order in terms of intersections of cosets (and make these  structures into coset geometries as defined by Tits \cite{asc,buc,bup,pa,tib,ti,rmw}). The analysis in \cite{esdiss,kom2} was carried out in terms of the conditions of Lemma~\ref{incchar}(b) and (\ref{incdef}).
\bigskip

\noindent
{\bf Acknowledgment.} I am grateful to the referees for their careful reading of the original manuscript and their helpful suggestions that have improved this article.

\end{document}